\documentclass[a4paper,11pt]{article}
\input{xypic}
\usepackage{relsize}
\usepackage{float,caption}
\usepackage{hyperref,bookmark,booktabs}%,natbib}
\usepackage{amsmath, amsthm, amssymb}
\usepackage[top=1.25in, bottom=1.25in, left=1.25in, right=1.25in]{geometry}
\usepackage{graphicx} %%% for figure
\usepackage{epstopdf} %% for figure in eps
\usepackage{mdframed}
\usepackage{ragged2e}
\usepackage{mathrsfs}
\usepackage[shortlabels]{enumitem}
\usepackage{tikz-cd}
\usepackage{setspace}
\usepackage{stmaryrd}

\newtheorem{thm}{Theorem}[section]

\newtheorem{lem}[thm]{Lemma}

\newtheorem{defn}[thm]{Definition}

\newtheorem{exam}[thm]{Example}

\theoremstyle{definition}
\def\Int{\operatorname{Int}}
\def\Cl{\operatorname{Cl}}

\def\T{\mathcal{T}}
\def\U{\mathcal{U}}

\begin{document}
\setcounter{page}{1}
\title{\href{https://doi.org/10.22105/jfea.2022.356442.1230}{In Press: Journal of Fuzzy Extension and Applications}
	\\
	\textbf{A note on somewhat fuzzy continuity}}
\author{Zanyar A. Ameen\\
\footnotesize{Department of Mathematics, College of Science, University of Duhok, Duhok 42001, Iraq}\\
{\tt  zanyar@uod.ac}\\}
\date{}
\maketitle
	
\begin{abstract}
Thangaraj and Balasubramanian introduced the so-called somewhat fuzzy semicontinuous and somewhat fuzzy semiopen functions. Two years later, the same authors defined two other types of functions called somewhat fuzzy continuous and somewhat fuzzy open without indicating connections between them. At first glance, we may easily conclude (from their definitions) that every somewhat fuzzy continuous (resp. open) function is somewhat fuzzy semicontinuous (resp. semiopen) but not conversely. In this note, we show that they are equivalent. We further prove that somewhat fuzzy continuous functions are weaker than fuzzy semicontinuous functions.
\end{abstract}
\textbf{Key words and phrases:}  fuzzy continuous, somewhat continuous, somewhat fuzzy continuous,  somewhat fuzzy semicontinuous, somewhat fuzzy open\\
\textbf{2020 MSC:} 54A40, 03E72, 54A05, 54C99.

\section{Introduction and Preliminaries}\label{sec1}
One of the fundamental concepts of all mathematics, particularly  topology and analysis, is continuity of functions, which means that "small" changes in the input result in "small" changes in the output. After the introduction of the fuzzy set, Chang defined the notion of fuzzy continuity \cite{chang} of functions between fuzzy topological spaces in 1968. Following that, several generalizations of fuzzy continuity were introduced. In 1981, Azad \cite{azad} presented some classes of generalized fuzzy continuous functions. Fuzzy semicontinuity is one of them that is weaker than fuzzy continuity. In 2001-2003, Thangaraj and Balasubramanian, respectively, introduced somewhat fuzzy semicontinuous \cite{T1} and somewhat fuzzy continuous functions \cite{T2}. They only discussed their relationship with fuzzy continuous functions. In the present note, we show that somewhat fuzzy semicontinuity and somewhat fuzzy continuity are the same, and somewhat fuzzy continuity is weaker than fuzzy semicontinuity. It is worth stating somewhat continuity appeared in different topological structures, (see, \cite{z2,gen,T2}). These classes of functions have a great role in characterizing (soft) Baire spaces \cite{z3,T3} and weakly equivalence spaces \cite{gen}.

Let $X$ be a universe (domain). A mapping $\mu$ from $X$ to the unit interval $\mathbb{I}$ is named a fuzzy set in $X$. The value $\mu(x)$ is called the degree of the membership of $x$ in $\mu$ for each $x\in X$. The support of $\mu$ is the set $\{x\in X:\mu(x)>0\}$. The complement of $\mu$ is, denoted by $\mu^c$ (or $1-\mu$ if there is no confusion), given by $\mu^c(x)=1-\mu(x)$ for all $x\in X$. Let $\{\mu_j:j\in J\}$ be a collection of fuzzy sets in $X$, where $J$ is any index set. Then $\bigvee\mu_j(x)=\sup\{\mu_j(x):j\in J\}$ and $\bigwedge\mu_j(x)=\inf\{\mu_j(x):j\in J\}$ for each $x\in X$.% A fuzzy set $\mu$ in $X$ is said to be a fuzzy point, denoted by $x_p$, with the support $x\in X$ and the value $p\in(0,1)$ if $\mu_x:X\to\mathbb{I}$ is the function defined as follows: for each $y\in X$, 	\[\mu_x(y)=\begin{cases}		p  &\text{ if }y=x;\\	0  &\text{ otherwise}.\end{cases}\] A fuzzy point $x_p$ said to be in $\mu$, denoted by $x_p\in\mu$, if $p\leq\mu(x)$. We may write $x\in\mu$ if no confusion causes. 

\begin{defn}\label{top}\cite{chang}
A collection $\T$ of fuzzy sets in $X$ is said to a fuzzy topology on $X$ if
	\begin{enumerate}[(i)]
		\item $0,1\in \T$; 
		\item $\mu\wedge\lambda\in T$ whenever $\mu,\lambda\in \T$; and
		\item $\bigvee\mu_j\in \T$ for any subcollection $\{\mu_j:j\in J\}$ of $\T$.
	\end{enumerate}
The pair $(X,\T)$ is called a fuzzy topological space. Members of $\T$ are called fuzzy open subsets of $X$ and their complements are called fuzzy closed sets. \end{defn}

Let $\lambda$ be a fuzzy set in $X$ and let $(X,\T)$ be a fuzzy topological space. The fuzzy interior of $\lambda$ is defined by $\Int(\lambda):=\sup\{\mu:\mu\leq\lambda, \mu\in T\}$  and the fuzzy closure of $\lambda$ is given by $\Cl(\lambda):=\inf\{\mu:\lambda\leq\mu, 1-\mu\in T\}$. The set $\lambda$ is called fuzzy semiopen \cite{azad} if $\lambda\leq\Cl(\Int(\lambda))$. The complement of every fuzzy semiopen set is fuzzy semiclosed. Evidently, every fuzzy open set is fuzzy semiopen but not the opposite (see, Example \ref{ex1}). The fuzzy semi-interior of $\lambda$ is defined by $\Int_s(\lambda):=\sup\{\mu:\mu\leq\lambda, \mu \text{ is fuzzy semiopen}\}$  and the fuzzy semi-closure of $\lambda$ is given by $\Cl_s(\lambda):=\inf\{\mu:\lambda\leq\mu, \mu\text{ is fuzzy semiclosed}\}$.

\begin{lem}\label{interior-lemma}\cite{fuzzy semi-int}
Let $\lambda$ be a fuzzy set in $X$ and let $(X,\T)$ be a fuzzy topological space. The following hold:
\begin{enumerate}[(i)]
	\item $\Int(\lambda)\leq\Int_s(\lambda)$.
	\item $\Cl_s(\lambda)\leq\Cl(\lambda)$.
\end{enumerate}
\end{lem}

\vspace{5mm}
To make our task easier, we introduce the following notions:
\begin{defn}
Let $\lambda$ be a fuzzy set in $X$ and let $(X,\T)$ be a fuzzy topological space. Then 
	\begin{enumerate}[(i)]
		\item $\lambda$ is called somewhat fuzzy open set if $\lambda=0$ or $\Int(\lambda)\ne 0$.
		\item $\lambda$ is called somewhat fuzzy semiopen set if $\lambda=0$ or $\Int_s(\lambda)\ne 0$.
	\end{enumerate}
\end{defn}
\begin{defn}\label{cont}
	A function $f$ from a fuzzy topological space $(X,\T)$ into another fuzzy topological space $(Y,\U)$ is said to be 
	\begin{enumerate}[(1)]
		\item fuzzy continuous \cite{chang} if $f^{-1}(\beta)\in\T$ for each $\beta\in \U$.
		\item fuzzy semicontinuous \cite{azad} if $f^{-1}(\beta)$ is fuzzy semiopen for each $\beta\in \U$.
		\item somewhat fuzzy continuous \cite{T2} if $f^{-1}(\beta)$ is somewhat fuzzy open for each $\beta\in \U$.
		\item somewhat fuzzy semicontinuous \cite{T1} if $f^{-1}(\beta)$ is somewhat fuzzy semiopen for each $\beta\in \U$.
	\end{enumerate} 
\end{defn}

\begin{defn}\label{open}
	A function $f$ from a fuzzy topological space $(X,\T)$ into another fuzzy topological space $(Y,\U)$ is said to be 
	\begin{enumerate}[(1)]
		\item fuzzy open \cite{wong} if $f(\alpha)\in\U$ for each $\alpha\in \T$.
		\item fuzzy semiopen \cite{azad} if $f(\alpha)$ is fuzzy semiopen for each $\alpha\in \T$.
		\item somewhat fuzzy open \cite{T2} if $f(\alpha)$ is somewhat fuzzy open for each $\alpha\in \T$.
		\item somewhat fuzzy semiopen \cite{T1} if $f(\alpha)$ is somewhat fuzzy semiopen for each $\alpha\in \T$.
	\end{enumerate} 
\end{defn}

For other undefined terminologies and notations stated in this note, we refer the readers to  \cite{chang,wong}.

%%%%%%%%%%%%%%%%%%%%%%%%%%%%%%%%%%%%%
\section{Main results}
We start by proving several lemmas that help us achieve our goal.
\begin{lem}\label{L1}
Let $\lambda$ be a non-zero fuzzy set in a fuzzy topological space $(X,\T)$. Then $\lambda$ is fuzzy semiopen iff $\Cl(\lambda)=\Cl(\Int(\lambda))$.
\end{lem}
\begin{proof}
Given a fuzzy semiopen set $\lambda$ such that $\lambda\ne 0$. Then $\lambda\leq\Cl(\Int(\lambda))$ and so $\Cl(\lambda)\leq\Cl(\Int(\lambda))$. On the other hand, we always have $\Int(\lambda)\leq \lambda$. Therefore $\Cl(\Int(\lambda))\leq \Cl(\lambda)$. Hence $\Cl(\lambda)=\Cl(\Int(\lambda))$.

Conversely, assume that $\Cl(\lambda)=\Cl(\Int(\lambda))$, but $\lambda\leq\Cl(\lambda)$ always, so $\lambda\leq\Cl(\Int(\lambda))$. Thus $\lambda$ is fuzzy semiopen. 
\end{proof}

\begin{lem}\label{L2}
Let $\lambda$ be a non-zero fuzzy set in the fuzzy topological space $(X,\T)$. If $\lambda$ is fuzzy semiopen, then $\Int(\lambda)\neq 0$.
\end{lem}
\begin{proof}
Assume if possible that $\lambda$ is a fuzzy semiopen set for which $\Int(\lambda)=0$, by Lemma \ref{L1}, $\Cl(\lambda)=0$ implies that $\lambda=0$, a contradiction.
\end{proof}

\begin{lem}\label{L3}
Let $\lambda$ be a fuzzy set in a fuzzy topological space $(X,\T)$. Then $\lambda$ is somewhat fuzzy open iff it is somewhat fuzzy semiopen.
\end{lem}
\begin{proof}
Wolg we assume $\lambda\ne 0$, otherwise, the equivalence of these concepts is clear. Since, by Lemma \ref{interior-lemma}, $\Int(\lambda)\leq\Int_s(\lambda)$, so this concludes that somewhat fuzzy openness of $\lambda$ implies somewhat fuzzy semiopenness. 

Conversely, suppose that $\lambda$ is somewhat fuzzy semiopen. By definition, for some $x\in\lambda$, there is a fuzzy semiopen set $\mu$ in $X$ such that $x\in\mu\leq\lambda$. By Lemma \ref{L2}, $\Int(\mu)\ne 0$. Therefore, for some $x\in\lambda$, one can find a fuzzy open set $\beta$ such that $x\in\beta\leq\mu\leq\lambda$. Hence, $\Int(\lambda)\ne 0$. This shows that $\lambda$ is somewhat fuzzy open.
\end{proof} 

From the above lemmas, perhaps a relationship among the stated classes of fuzzy open sets can be concluded. 
\begin{figure}[H]
	\begin{center}
		\begin{tikzcd}[sep=small, arrows=Rightarrow]
			\text{fuzzy open}\arrow{r}&\text{fuzzy semiopen}\arrow{r}&\text{somewhat fuzzy open}\arrow[equal]{r}&\text{somewhat fuzzy semiopen}
		\end{tikzcd}
		\caption*{\quad Diagram 1: Relationships between generalized fuzzy open sets}\label{D1}
	\end{center}
\end{figure}

None of the arrows in the diagram above can be reversed, as shown in the following example:
\begin{exam}%\cite{z1,azad}
\label{ex1}
Let $\mu, \lambda, \sigma ,\alpha, \beta$ be fuzzy sets on $\mathbb{I}$ defined as:
	\[
	\mu(x)=\begin{cases}
		0, &\text{ if } 0\leq x\leq 1/2\\
		2x-1, &\text{ if } 1/2\leq x\leq 1,
	\end{cases}
	\]
	\[
	\lambda(x)=\begin{cases}
		1, &\text{ if } 0\leq x\leq 1/4\\
		-4x+2, &\text{ if } 1/4\leq x\leq 1/2\\
		0, &\text{ if } 1/2\leq x\leq 1,
	\end{cases}
	\]
	\[
	\sigma(x)=\begin{cases}
		1, &\text{ if } 0\leq x\leq 1/4\\
		-4x+2, &\text{ if } 1/4\leq x\leq 1/2\\
		2x-1, &\text{ if } 1/2\leq x\leq 1,
	\end{cases}
	\]
		\[
	\alpha(x)=\begin{cases}
		0, &\text{ if } 0\leq x\leq 1/4\\
		4x/3-1/3, &\text{ if } 1/4\leq x\leq 1,
	\end{cases}
	\]
and 
\noindent
	\[
\beta(x)=\begin{cases}
	2x, &\text{ if } 0\leq x\leq 1/2\\
	1, &\text{ if } 1/2\leq x\leq 1.
\end{cases}
\]
	
The family $\T=\{0,\mu, \lambda, \sigma, 1\}$ forms a fuzzy topology on $\mathbb{I}$. The set $\alpha$ is fuzzy semiopen but not fuzzy open (see, Example 4.5 in \cite{azad}). One can easily check that the set $\beta$ is somewhat fuzzy open but not fuzzy semiopen.
\end{exam}

Thangaraj and Balasubramanian showed that fuzzy continuity (openness) of a function implies somewhat fuzzy continuity (openness), but in reality more than that is true (from Diagram 1).
\begin{figure}[H]
	\centering
	\begin{tikzcd}[sep=small]%,row sep=1cm, column sep=1.5cm,every cell/.append style={align=center}]
		& 	
		\begin{tabular}{c}
		fuzzy continuous \\(fuzzy open)
		\end{tabular}
\hspace{-.2cm}\implies\hspace{-1cm}
		& 
		\begin{tabular}{c}
			fuzzy semicontinuous\\ (fuzzy semiopen)
		\end{tabular}
\hspace{-.2cm}\implies\hspace{-1cm}
		&
		\begin{tabular}{c}
		somewhat fuzzy continuous\\ (somewhat fuzzy open)
		\end{tabular}
	\end{tikzcd}
	\caption*{\quad Diagram 2: The connections between the generalized fuzzy functions}\label{D2}
\end{figure}

\begin{thm}
For a function $f:(X, \T)\rightarrow (Y,\U)$, the following statements hold:
\begin{enumerate}[(1)]
	\item $f$ is somewhat fuzzy continuous iff it is somewhat fuzzy semicontinuous.
	\item $f$ is somewhat fuzzy open iff it is somewhat fuzzy semiopen.
\end{enumerate}
\end{thm}
\begin{proof}
	The Lemma \ref{L3} guarantees that $\Int(\mu)\ne 0$ iff  $\Int_s(\mu)\ne 0$ for every fuzzy set $\mu$ and the proof can be concluded.
\end{proof}

%(1) Suppose $f$ is a somewhat fuzzy continuous function. Let $\beta\in\U$. By assumption, $f^{-1}(\beta)$ is somewhat fuzzy open, i.e., $\Int(f^{-1}(\beta))\ne 0$. But, $\Int(f^{-1}(\beta))\leq\Int_s(f^{-1}(\beta))$ by Lemma \ref{interior-lemma}. This implies $\Int_s(f^{-1}(\beta))\ne 0$ and so $f$ is somewhat fuzzy semicontinuous.Conversely, assume $f$ is somewhat fuzzy semicontinuous and let $\beta\in\U$. By assumption and Lemma \ref{L3}, $\Int(f^{-1}(\beta))=\Int_s(f^{-1}(\beta))\ne 0$. Thus $\Int(f^{-1}(\beta))\ne 0$ which proves that $f$ is somewhat fuzzy continuous.(2) Similar to (1).

We finish this note by pointing that the same result holds true for the aforementioned functions in general and soft topology. 

%%%%%%%%%%%%%%%%%%%%%%%%%5 BiB  %%%%%%%%%%%%%%%%%%5
%\bibliographystyle{abbrv}
%\bibliography{infrafuzzy}

\begin{thebibliography}{999}

\bibitem{z1}Z. A. Ameen, T. Al-shami, A. A. Azzam, A. Mhemdi, A novel fuzzy structure: Infra-fuzzy topological spaces, Journal of Function Spaces. 2022, Article ID 9778069, 2022, 1-11.

\bibitem{z2}
Z. A. Ameen, B.A. Asaad, T.M. Al-shami, Soft somewhat continuous and soft somewhat open functions, TWMS Journal of Applied and Engineering Mathematics 13 (2) (2023), 792-806

\bibitem{z3}
Z. A. Ameen, A. B. Khalaf, The invariance of soft Baire spaces under soft weak functions, Journal of Interdisciplinary Mathematics, 25 (5) (2022), 1295-1306.

\bibitem{z4}
B. A. Asaad, T. M. Al-shami, Z. A. Ameen, On soft somewhere dense open functions and soft Baire spaces, Iraqi Journal of Science 64 (1) (2023), 373-384.

	
\bibitem{azad}
K. K. Azad, On fuzzy semi-continuity, fuzzy almost continuity and fuzzy weakly continuity, J. Math. Anal. Appl. 82 (1981) 14--32.


\bibitem{chang}
C. Chang, Fuzzy topological spaces, Journal of Mathematical Analysis and Applications, 24 (1)  (1968) 182--190. %\href{https://doi.org/10.1016/0022-247X(68)90057-7}{https://doi.org/10.1016/0022-247X(68)90057-7}

\bibitem{gen}
K. R. Gentry, H. B. Hoyle III, Somewhat continuous functions, Czechoslovak Mathematical Journal, 21 (1) (1971), 5--12.

%\bibitem{P} P. Pao-Ming and L. Ying-Ming, Fuzzy topology. I. Neighborhood structure of a fuzzy point and Moore-Smith convergence, Journal of Mathematical Analysis and Applications 76 (2) (1980) 571-599.

\bibitem{T1}
G. Thangaraj, G. Balasubramanian, On somewhat fuzzy semicontinuous functions, Kybernetika, 37 (2) (2001) 165--170.

\bibitem{T2}
G. Thangaraj, G. Balasubramanian, On somewhat fuzzy continuous functions, J. Fuzzy Math., 11 (2) (2003) 725--736.

\bibitem{T3}
G. Thangaraj, R. Palani, Somewhat fuzzy continuity and fuzzy Baire spaces, Ann. Fuzzy. Math. and Info, 12 (1) (2016)  75--8.

%\bibitem{lowen} R. Lowen, Fuzzy topological spaces and fuzzy compactness, Journal of Mathematical analysis and applications,  56 (3) (1976) 621-633 \href{https://doi.org/10.1016/0022-247X(76)90029-9}{https://doi.org/10.1016/0022-247X(76)90029-9}

\bibitem{wong} C. K. Wong, Fuzzy topology: Product and quotient theorems, Journal of Mathematical Analysis and Applications, 45 (1974) 512-521.

\bibitem{fuzzy semi-int}
T. H. Yalvaç, Semi-interior and semi-closure of a fuzzy set,  Journal of Mathematical Analysis and Applications, 132 (1988) 356-36.
\end{thebibliography}
\end{document}